\documentclass[reqno,a4paper,11pt]{amsart}
\usepackage{color}
\usepackage{amsmath,amssymb,amsfonts,amsthm,enumerate}
\usepackage{hyperref,mathrsfs,accents}
\usepackage[T1]{fontenc}
\usepackage[active]{srcltx}
\usepackage{epsfig}
\usepackage{graphicx}
\usepackage{cite}
\usepackage{tikz}
\usepackage{pgfplots}
\usepackage{tikz-3dplot}
\usepackage{subcaption}
\usepackage{pgfplots}
\usepackage{adjustbox}
\usepackage{bigints}
\usetikzlibrary{patterns,backgrounds}

\tdplotsetmaincoords{60}{115}
\pgfplotsset{compat=newest}

\catcode`@=11 \@addtoreset{equation}{section} \catcode`@=12

\newtheorem{theorem}{Theorem}[section]
\newtheorem{lemma}[theorem]{Lemma}

\newcommand{\N}{\mathbb{N}}
\newcommand{\R}{\mathbb{R}}

\newcommand{\frX}{\mathfrak{X}}

\newcommand{\Hau}{{\mathcal H}} % Misura di Hausdorff
\renewcommand{\dim}{\mathop{\rm dim}} %dimension
\newcommand{\de}{\partial}

\newcommand{\sgn}{\mathop{\mathrm{sgn}}}

\newcommand{\difsim}{\Delta}

\newcommand{\restrict}{\mathbin{\vrule height 1.4ex depth 0pt width
		0.13ex\vrule height 0.13ex depth 0pt width 1.3ex}\,}

\newcommand{\cA}{\mathcal{A}}

\newcommand{\cI}{\mathcal{I}}

\renewcommand{\Subset}{\subset\!\subset}

\newcommand{\p}{P}

\renewcommand{\phi}{\varphi}
\renewcommand{\epsilon}{\varepsilon}

\AtBeginDocument{
	\let\div\relax
	\DeclareMathOperator{\div}{div}
}

\definecolor{grey}{rgb}{.7,.7,.7}
\definecolor{evidGP}{rgb}{0,0,1}
\definecolor{evidG}{rgb}{1,0,0}

\author{Giacomo Vianello}
\address{Czech Academy of Sciences (Prague - Czech Republic)}
\email{vianello@utia.cas.cz}

\thanks{The author thanks Gian Paolo Leonardi for constructive discussions on the topics of the paper. The author has been supported by GNAMPA (INdAM) Project 2025: ``Structures of sub-Riemannian hypersurfaces in Heisenberg groups''; GNAMPA (INdAM) Project 2026: ``Variational, Geometric, and Analytic Perspectives on Regularity''.}



\title[Minimality of free-boundary...]{Minimality of free-boundary axial hyperplanes in high dimensional circular cones via calibration}

\begin{document}
	\begin{abstract}
		Consider an $(n+1)$-dimensional circular cone with opening angle $\alpha \in (0,\pi)$. Using a free-boundary adaptation of the classical calibration method, we prove that, for $n \geq 4$, there exists a threshold $\bar{\alpha}(n) \in (0,\pi)$ such that if $\alpha \geq \bar{\alpha}(n)$, that is, the cone is wide enough, the intersection of the cone with an axial hyperplane is area-minimizing with respect to free-boundary variations inside the cone. This provides a counterexample to a recent Vertex-skipping Theorem proved by the author in collaboration with G.P. Leonardi, at least for $n\geq4$.
	\end{abstract}
	\maketitle
    \section{Introduction}
    Consider a volume $v > 0$ of liquid inside a container, here identified with a Lipschitz open set $\Omega \subset \R^3$. According to Gauss' formulation, in the absence of gravity, the configuration $E$ assumed by the liquid minimizes the free-energy functional
    \begin{equation} \label{eq:cap_ener}
    	\p(F;\Omega) + \gamma \, \p(F;\de \Omega)
    \end{equation}
    among all sets $F \subset \Omega$ such that $\p(F;\Omega) < +\infty$ and satisfying the volume constraint $|F| = v$.
    Here, $\gamma \in [-1,1]$ is the so-called relative adhesion coefficient between the liquid and $\de \Omega$, and $\p(F;U)$ stands for the relative perimeter of $F$ in a Borel set $U$ 
    % , and $F$ is said to have finite perimeter in $\Omega$ provided $\p(F;\Omega) < +\infty$
     (see \cite{maggi2012sets,Fin86book} for more details).

    From a mathematical point of view, given an open set $\Omega \subset \R^{n+1}$, a well-established internal regularity theory for minimizers of \eqref{eq:cap_ener} is available in the literature. In fact, let $E$ be a minimizer of \eqref{eq:cap_ener} and indicate by $\de^{\ast} E$ its reduced boundary.
    Letting $M := \overline{\de^{\ast} E \cap \Omega}$, it is known that $M \cap \Omega$, that is, the internal portion of $M$, is formed by the union of a smooth hypersurface and a singular set $\Sigma$ whose Hausdorff dimension $\dim_{\Hau} \Sigma$ does not exceed the critical value $n-7$, being in particular nonempty only when $n \geq 7$; see \cite{DePhilippisMaggi2015} and the references therein.
    
    Much less is known regarding the boundary regularity for minimizers of \eqref{eq:cap_ener}, that is, the regularity of the free-boundary interface $M \cap \de \Omega$. One of the first general regularity results, which holds for smooth containers $\Omega \subset \R^3$, is due to Taylor \cite{JTaylor_BR_1977}, who proved that $M \cap \de \Omega$ does not exhibit singularities. This result has been generalized much later by De Philippis and Maggi in \cite{DePhilippisMaggi2015} for minimizers of anisotropic capillarity functionals. In the isotropic case, which includes the energy functional \eqref{eq:cap_ener}, their result ensures that, if $\de \Omega$ is $C^{1,1}$, then $M$ is the union of a $C^{1,1/2}$ hypersurface with boundary and a singular set $\Sigma$ with the property that $\dim_{\Hau}(\Sigma \cap \de \Omega) \leq n-3$. Further improvements on the Hausdorff dimension of $\Sigma \cap \de \Omega$ have been provided in \cite{CEL_Impr_reg_cap_2025} when $\gamma$ is close to either $0$, $1$ or $-1$. %See also \cite{DP_F_M_2024} for a regularity result for $\Lambda$-minimizers of \eqref{eq:cap_ener} under the hypothesis that $\Omega$ is an half-space and the wet region of competitors is constrained inside an open subset of $\de \Omega$; \cite{PTV_capcones_2025} for stability and rigidity properties of singular capillary cones arising in the capillarity setting.
    See also \cite{DP_F_M_2024} for a regularity result on $\Lambda$-minimizers of \eqref{eq:cap_ener} under the assumption that $\Omega$ is a half-space and the wet region of competitors is constrained within an open subset of $\partial \Omega$; and \cite{PTV_capcones_2025} for stability and rigidity properties of singular capillary cones. %arising in the capillarity setting.

    Focusing now on the special case $\gamma = 0$, corresponding to the relative perimeter in $\Omega$, something stronger regarding the regularity of $M \cap \de \Omega$ was proved by Gr\"uter-Jost \cite{GruterJost1986allard} and Gr\"uter \cite{Gruter1987,Gruter_opt_reg_1987} in the setting of varifolds, under the assumption that $\de \Omega$ is $C^2$.  
    In particular, in \cite{GruterJost1986allard}, the authors demonstrate that if a varifold $M$ is close to a free-boundary hyperplane in $\Omega$, then $M$ is $C^{1,\beta}$-graphical with respect to this hyperplane. Subsequently, in \cite{Gruter_opt_reg_1987}, the author exploits the latter result to show that, as in the interior case, the Hausdorff dimension of the singular set of the free-boundary satisfies
    %\begin{equation} \label{eq:bound_bdry_reg}
    $\dim_{\Hau}(\Sigma \cap \de \Omega) \leq n-7$.
    %\end{equation}
    
    The common denominator of the boundary regularity results mentioned above, both for $\gamma \neq 0$ and $\gamma = 0$, is that the container $\Omega$ must be smooth. The literature concerning boundary regularity becomes sparser when one looks for results involving nonsmooth containers $\Omega$. In this situation, the main available studies typically consider specific nonsmooth domains or classes of domains. This is the case, for example, in the series of works \cite{hildebrandt1997wedge1,hildebrandt1997wedge2,hildebrandt1999wedge3,hildebrandt1999wedge4} by Hildebrandt and Sauvigny, where the authors perform a thorough study of minimal surfaces in wedges (i.e., three-dimensional domains obtained as the intersection of two half-spaces); in particular, in \cite[p. 73]{hildebrandt1999wedge4}, they prove a regularity result for free-boundary minimal graphs on a plane orthogonal to the wedge.
    
    The latter result has been recently generalized by Edelen and Li in \cite{EdelenLi2022}. 
    In that work, they prove a boundary regularity theorem in the same spirit as Gr\"uter-Jost \cite{GruterJost1986allard}, holding for a varifold $M$ under the assumption that $\Omega$ is locally polyhedral and that $M$ is close to a horizontal plane (for instance, if $\Omega = W^2 \times \R$ is a wedge, horizontal planes are those containing $W^2 \times \{0\}$).
    In the same work, Edelen and Li also prove that $\dim_{\Hau}(\Sigma \cap \de \Omega) \leq n-2$, with a sharper estimate holding when the dihedral angles of their cone models are nonobtuse; see \cite[Theorem 1.2]{EdelenLi2022}. 
    In the setting of their polyhedral domains, the singular set $\Sigma$ is defined as the complementary set of the points around which the interface $M$ is the graph of a $C^{1,\beta}$ function over a horizontal plane.
    
    In this paper, we prove that for $n \geq 4$ and for circular cones with sufficiently large aperture, axial hyperplanes are area-minimizing with respect to free-boundary variations.
    This result fits naturally into some very recent contributions by the author in collaboration with Leonardi \cite{LeoVia1-2024,LeoVia2-2024} concerning boundary regularity in nonsmooth domains. In the latter, we establish a boundary monotonicity formula that holds for almost minimizers of the relative perimeter under mild regularity assumptions on $\de \Omega$, which is then employed to prove a minimizing-cone property for a blow-up sequence at a boundary point of $\Omega$.
    In \cite{LeoVia1-2024}, the authors prove a Vertex-skipping Theorem, which states that the boundary of any almost minimizer of the relative perimeter $E$ in a convex domain $\Omega \subset \R^3$ cannot contain isolated singularities (say, vertices) of $\de \Omega$. 
    We highlight that, after a dimension-reduction argument, this result combined with \cite[Theorem 1.1]{EdelenLi2022} implies the bound $\dim_{\Hau}(\Sigma \cap \de \Omega) \leq n-3$, improving the estimate given in \cite[Theorem 1.2 (1)]{EdelenLi2022}.
    %The study of the behavior of $E$ near a singularity of the container is also relevant from the perspective of free-boundary minimal surfaces, since the interface
    %$M = \overline{\partial^{\ast} E \cap \Omega}$
    %is a free-boundary minimal hypersurface in $\Omega$. We refer to \cite{Carlotto2019} and the references therein for background on free-boundary minimal surfaces.
    
    In this framework, the main result proved here provides a counterexample to the validity of Vertex-skipping results in higher dimensions.
    The question of the validity of the Vertex-skipping for $n \geq 3$ was already tackled in \cite{LeoVia_stability}. In that work, using a suitable Lipschitz flow of deformations for a free-boundary surface in a Lipschitz container $\Omega$, which in particular allows $\de \Omega$ (including its singular set) to move, we establish a stability result for axial hyperplanes in circular cones, stated as follows.
    Given an integer $n \geq 2$, and for $\lambda > 0$, consider the circular cone
    \begin{equation*}
    	\Omega_{\lambda} := \left\{ (x,t) \in \R^{n+1} \, : \, t > \lambda \sqrt{x_1^2 + \dots + x_n^2} \, \right\} .
    \end{equation*}
    Note that the opening angle $\alpha_{\lambda} \in (0,\pi)$ of $\Omega_{\lambda}$ is related with $\lambda$ by $\lambda = \cot(\alpha_{\lambda}/2)$. Then, if $n = 2$, a hyperplane of $\R^{3}$ containing the axis of the cone (up to rotations, $x_1 = 0$) is unstable in $\Omega_{\lambda}$, for every $\lambda > 0$. This is in accordance with the aforementioned Vertex-skipping Theorem.
    In partial contrast, when $n \geq 3$, we show that there exists a threshold $\lambda^{\ast}(n) > 0$ such that if $0 < \lambda \leq \lambda^{\ast}(n)$, i.e., the circular cone has a sufficiently large opening angle $\alpha \geq \alpha^{\ast}(n) := \alpha_{\lambda^{\star}(n)}$, the axial hyperplane is stable in $\Omega_{\lambda}$, whereas for $\lambda > \lambda^{\ast}(n)$, or equivalently if $\alpha < \alpha^{\ast}(n)$, it is unstable.
    In this paper, we intend to analyze whether, for $n \geq 3$ and $0<\lambda \leq \lambda^{\ast}(n)$, an axial hyperplane is merely stable or even area-minimizing in $\Omega_{\lambda}$. We will provide a basically positive answer for $n \geq 4$.
    
    We say that a measurable set $E \subset \R^{n+1}$ has locally finite perimeter provided $\p(E;U) < + \infty$ for every $U \Subset \R^{n+1}$, see \cite{Giu84book,maggi2012sets}. For $(x,t) \in \R^{n+1}$ and $r > 0$, we denote by $B_r(x,t)$ the ball of radius $r$ centered at $(x,t)$ %(omitting $(x,t)$ when $x = 0$ and $t = 0$).
    (setting $B_r := B_r(0,0)$).
    Our main result is the following:
    
    %\smallskip
    \begin{theorem} \label{thm:min}
    	For $n \geq 4$, consider the threshold parameter
    	\begin{equation}
    		\label{eq:def_barlambda} \bar{\lambda}(n) := \dfrac{1}{2} \dfrac{n-3}{\sqrt{n-2}} ,
    	\end{equation}
    	and assume that $0 < \lambda \leq \bar{\lambda}(n)$. Let also
    	\begin{equation*}
    		E := \{ (x,t) \in \R^{n+1} \, : \, x_1 > 0 \}.
    	\end{equation*}
    	Then $E$ is a minimizer of the relative perimeter in $\Omega_{\lambda}$, that is, given a locally-finite perimeter set $F \subset \R^{n+1}$ satisfying $E \difsim F \Subset B_R$, for some $R>0$, one has
    	\begin{equation} \label{eq:min_ineq}
    		\p(E;\Omega_{\lambda} \cap B_R) \leq \p(F;\Omega_{\lambda} \cap B_R) .
    	\end{equation}
    \end{theorem}
    
    \medskip
    \noindent
    In other words, if the opening angle $\alpha$ of the circular cone satisfies $\alpha \geq \bar{\alpha}(n) := \alpha_{\bar{\lambda}(n)}$, then $H_{\lambda} := \de E \cap \Omega_{\lambda}$ is a free-boundary area-minimizing surface inside the cone.
    
    We emphasize that the above theorem provides a definitive counterexample to the validity of a Vertex-skipping Theorem such as \cite[Theorem 1.1]{LeoVia1-2024} in higher dimensions, but only when $n \geq 4$.  
    The case $n = 3$ is a limiting situation where, in view of \cite{LeoVia_stability}, provided $0 < \lambda \leq \lambda^{\ast}(n)$, $H_{\lambda}$ is stable with respect to free-boundary variations in $\Omega_{\lambda}$, but our minimality result, Theorem \ref{thm:min}, does not apply.  
    We also remark that, despite $\Omega_{\lambda}$ not being a polyhedral cone, Theorem \ref{thm:min} appears to suggest that the dimensional bound for the singular set of the free-boundary proved in \cite[Theorem 1.2 (1)]{EdelenLi2022} cannot be improved beyond the threshold $n-4$. In fact, when $n = 4$ and $0 < \lambda \leq 1/2\sqrt{2}$, the result above implies that $\de E \cap \de \Omega_{\lambda} = \{ 0 \}$, and so $0 \in \Sigma \cap \de \Omega_{\lambda}$.
    %and, since $0$ is singular, the boundary singular set $\Sigma \cap \de \Omega_{\lambda}$ proves to be nonempty.
    
    The proof of Theorem \ref{thm:min} is based on a calibration argument. 
    The calibration method was originally introduced by Bombieri, De Giorgi, and Giusti in their celebrated work \cite{BDG69} to prove the minimality of Simon's cone in $\R^8$. Since then, several authors have contributed by proposing new versions of the original method \cite{Lawson_1972,Federer_1975,Massari_Miranda_1983,Mackenzie_1987,Lawlor_Angle_1989,Brakke_hypercubes_1991,LM_Paired_1994,Davini_2004,DPh_Pao_2009} or by developing various generalizations to different contexts, such as the Steiner tree problem \cite{Marchese_Massaccesi_2016,MOV_NumCal_Stein_2019,Pluda_Pozzetta_2023} and the Mumford-Shah functional \cite{AmbrosioFuscoPallara2000,ABD_calibration_note,Mora_Morini_2001}. See \cite[\S 6.5]{Morgan_GMT_2009} for a more comprehensive overview on calibrations.

    Here, the calibration method is adapted to the free-boundary setting in the sense that we require the calibrating vector field to be tangent to $\de \Omega_{\lambda}$.
    Our calibration argument, explained in detail in Section \ref{sec:Min_E}, proceeds roughly as follows. For $\lambda > 0$, let $S_{\lambda} := \de \Omega_{\lambda} \setminus \{ 0 \}$, and denote by $e_1,\dots,e_n,e_t$ the canonical basis of $\R^{n+1}$.
    First, we construct a divergence-free vector field $Z$, defined outside a $2$-dimensional subset of $\overline{\Omega}_{\lambda}$, such that $Z|_{H_{\lambda}} = e_1$, $|Z| \leq 1$, and $Z|_{S_{\lambda}}$ is tangent to $S_{\lambda}$.
    Second, we apply the Divergence Theorem to $Z$ on the region $(E \difsim F)\cap\Omega_{\lambda}$. On the one hand, the fact that $Z$ is tangent to $S_{\lambda}$ allows us to disregard the contribution arising from the area of $(E \difsim F)\cap\de \Omega_{\lambda}$. On the other hand, the condition $\div Z = 0$ enables us to compare $\p(E;\Omega_{\lambda} \cap B_R)$ with $\p(F;\Omega_{\lambda} \cap B_R)$. We stress that here, as remarked above, the vector field $Z$ is not defined on the entire closed cone $\overline{\Omega}_{\lambda}$, but only outside a $2$-dimensional subset of $\overline{\Omega}_{\lambda}$. This forces us to perform an approximation argument in order to apply the Divergence Theorem to $(E \difsim F) \cap \Omega_{\lambda}$ and infer the desired inequality~\eqref{eq:min_ineq}; see Section~\ref{sec:Min_E}.
    
    The existence of the vector field $Z$ will be exhibited in Section \ref{sec:constr_cal}, and its construction will be carried out in two steps. We begin by calibrating the $(n-1)$-dimensional surface $S^0_{\lambda} := H_{\lambda} \cap S_{\lambda}$ in $S_{\lambda}$; that is, we build a divergence-free vector field $Y$, defined on a suitable subset $S_{\lambda}'$ of $S_{\lambda}$ and tangent to $S_{\lambda}$, such that $Y|_{S^0_{\lambda}} = e_1$ and $|Y| \leq 1$.
    In a slightly different formulation via differential forms, this was essentially done by Morgan in \cite{morgan2002area}, using ideas from Lawlor's PhD thesis \cite{Lawlor_Cones_1991}. Their construction basically reduces the problem to finding a solution to a suitable differential inequality subject to appropriate initial conditions (cf. the proof of Theorem \ref{thm:excal_Slambda}). Such a solution exists for $n \geq 4$ and $0 < \lambda \leq \bar{\lambda}(n)$, thereby explaining the dimensional gap occurring for $n = 3$.
    Second, we extend $Y$ to (a suitable portion of) $\Omega_{\lambda}$ by vertical projection onto $S_{\lambda}$: we will show that, due to the structure of the metric on $S_{\lambda}$, the vector field $Z$ obtained in this way is, in particular, divergence-free in its domain (cf. the proof of Theorem \ref{thm:excal}).
    
    \subsection{Organization of the paper.} In Section \ref{sec:Not_Prel} we provide some preliminary notation and notions concerning differentiable manifolds (especially Riemannian ones) and sets of finite perimeter. In Section \ref{sec:setting} we introduce further notation and perform some basic computations needed in the subsequent parts of the paper. Then, Section \ref{sec:constr_cal} is devoted to the construction of the calibration $Z$, whose main properties are collected in Theorem \ref{thm:excal}. Finally, in Section \ref{sec:Min_E} we present the calibration argument; namely, we show how the existence of the vector field $Z$ from Theorem \ref{thm:excal} can be used to demonstrate Theorem \ref{thm:min}.

	%%%%%
	\section{Preliminary notions and basic notation}
	\label{sec:Not_Prel}
	
	\subsection{Manifolds and $k$-forms.}
	Given a $n$-dimensional manifold $M$, we denote by $TM$, $T^{\ast}M$ the tangent and cotangent bundles of $M$ respectively. For any $k \geq 1$, we define the bundle $\bigwedge^k T^{\ast} M$ as the disjoint union of the vector spaces $\bigwedge^k (T_p M)^{\ast}$ (and we observe that $\bigwedge^{1} T^{\ast} M = T^{\ast}M$). When $k = 0$, we identify $\bigwedge^0 T^{\ast} M$ with $C^{\infty}(M)$. A function $\omega: M \to \bigwedge^k T^{\ast} M$ with the property that $\omega(p) \in \bigwedge^k (T_p M)^{\ast}$, for all $p \in M$, is called $k$-form. Given a local chart $\phi = (x_1,\dots,x_n)$ for $M$, we denote by $\de_i$ and $dx_i$, for $1 \leq i \leq n$, the local frames induced by $\phi$ on $TM$ and $T^{\ast}M$ respectively.
	%Thus, following what mentioned in subsection \ref{subsec:Multlin_Alg}, we can introduce a frame for $\bigwedge^k T^{\ast}M$, that is composed by the $k$-forms $dx_{\alpha}$, for all multi-indices $\alpha \in A_k$. Employing this frame, any $k$-form $\omega$ can be represented as
	We introduce the following multilinear algebra notation: given a basis $u_1,\dots,u_n$ of a $n$-dimensional vector space $V$ and a multi-index $\alpha = (\alpha_1,\dots,\alpha_{k}) \in \N^k$ with $1 \leq \alpha_1 < \dots < \alpha_k \leq n$, we set
	\begin{equation*}
		u_{\alpha} = u_{\alpha_1} \wedge\dots\wedge u_{\alpha_k} .
	\end{equation*}
	Moreover, fixing $1 \leq i < j \leq k$, we define
	\begin{equation} \label{eq:multlin_not}
		\begin{split}
			& \, \, \, u^i_{\alpha} = u_{\alpha_1}\wedge\dots\wedge u_{\alpha_{i-1}} \wedge u_{\alpha_{i+1}} \wedge \dots \wedge u_{\alpha_k} \\
			u^{i,j}_{\alpha} = u_{\alpha_1} & \wedge \dots\wedge u_{\alpha_{i-1}} \wedge u_{\alpha_{i+1}} \wedge \dots \wedge u_{\alpha_{j-1}} \wedge u_{\alpha_{j+1}} \wedge \dots \wedge u_{\alpha_k} .
		\end{split}
	\end{equation}
	Any $k$-form can be then represented by
	\begin{equation*} %\label{eq:loccoordkform}
		\omega = \sum_{\underset{1 \leq \alpha_1 < \dots < \alpha_k \leq n}{\alpha = (\alpha_1,\dots,\alpha_k) \in \N^k}} \omega_{\alpha} \, dx_{\alpha},
	\end{equation*}
	for suitable scalar functions $\omega_{\alpha} : M \to \R$.
	We indicate by $\frX(M)$, $\cA^k(M)$ the space of the vector fields on $M$ and the space of the $k$-forms on $M$, respectively.
	
	\subsection{Riemannian manifolds.}
	\label{subsec:Riem_Man}
	Let $(M,g)$ be a Riemannian manifold. Given two vectors $v,w \in T_p M$, we set $\langle v, w \rangle_g := g_p(v,w)$, $|v|_g := \sqrt{g_p(v,v)}$, with the convention that, if the subscript $g$ is omitted, we are referring to the Euclidean scalar product.
	The metric $g$ can be expressed in terms of the local frame for $TM$, by introducing the $\R$-valued functions
	\begin{equation*}
		g_{i,j}(p) := \langle \de_i(p), \de_j(p) \rangle_g \, , \qquad \text{for $1 \leq i,j \leq n$.}
	\end{equation*}
	With a little abuse, we denote by $g$ the symmetric, positive-definite matrix $g = (g_{i,j})_{1 \leq i,j \leq n}$, while we indicate by $\tilde{g} = (\tilde{g}_{i,j})_{1 \leq i,j \leq n}$ the inverse of $g$.
	The metric $g$ on $M$ can be extended to $\bigwedge^k T^{\ast} M$, for all $k \geq 1$. When $k=1$, we set
	\begin{equation*}
		\langle dx_i, dx_j \rangle_g := \tilde{g}_{i,j} \, , \qquad \text{for $1 \leq i,j \leq n$,}
	\end{equation*}
	and, when $k \geq 2$, given two multi-indices $\alpha= (\alpha_1,\dots,\alpha_k), \beta =(\beta_1,\dots,\beta_k) \in \mathbb{N}^k$, we define
	\begin{equation} \label{eq:metric_forms}
		\begin{split}
			\langle dx_{\alpha}, dx_{\beta} \rangle_g & := \det (\langle dx_{\alpha_i}, dx_{\beta_j} \rangle_g)_{1 \leq i,j \leq k} \\
			& = \det(\tilde{g}_{\alpha_i,\beta_j})_{1 \leq i,j \leq k} .
		\end{split}
	\end{equation}
	These definitions can be then extended to $\bigwedge^k T^{\ast}M$ by multi-linearity.
	By employing the metric structure of a Riemannian manifold, one can establish an identification between the tangent and the cotangent bundles.
	%We denote by $\cdot^\flat: TM \to T^{\ast} M$ and $\cdot^\#: T^{\ast} M \to TM$ the so-called \emph{musical isomorphisms}, which provide $1$-to-$1$ correspondences among vector fields and $1$-forms on $M$. 
	In local coordinates, this identification acts as follows: given a vector field $X = \sum_{i=1}^n X_i \de_i \in TM$ one can consider the $1$-form $X^{\flat}$ defined by %$\flat X = \sum_{i=1}^n \omega_i dx_i \in T^{\ast} M$ where 
	\begin{equation*}
		X^{\flat} = \sum_{i,j = 1}^{n} g_{i,j} X_i \, dx_j \, ;
	\end{equation*}
	conversely, given a $1$-form $\omega = \sum_{i=1}^n \omega_i \, dx_i \in T^{\ast} M$, one defines the vector field
	\begin{equation*}
		\omega^{\#} = \sum_{i,j = 1}^{n} \tilde{g}_{i,j} \, \omega_i \, \de_j \, .
	\end{equation*}
	Note that $|X^{\flat}|_g = |X|_g$, $|\omega^{\#}|_g = |\omega|_g$, and also $X^{\flat\#} = X$, $\omega^{\#\flat} = \omega$.
	
	Let now $M$ and $N$ be Riemannian manifolds and consider a diffeomorphism $F : M \to N$. Given a vector field $X \in \frX(M)$, the push-forward of $X$ with respect to $F$ is a vector field on $N$ defined by
	\begin{equation*}
		F_{\ast} X (q) = dF_{F^{-1}(q)}(X(F^{-1}(q))) \, , \qquad \text{for all $q \in N$.}
	\end{equation*}
	We denote by $\div^M X$ the divergence of a vector field $X$ on a Riemannian manifold $M$, 
	and we omit the superscript $M$ when referring to the Euclidean divergence. One can easily verify that, if $F$ is an isometry,
	\begin{equation} \label{eq:prop_divM}
		\div^N F_{\ast} X \, (q) = \div^M X \, (F^{-1}(q)) , \qquad \text{for every $q \in N$.}
	\end{equation}
	We conclude by stating the following formula for the divergence on $M$: if $X$ is expressed in local coordinates by $X = \sum_{i=1}^n X_i \de_i$, we have
	\begin{equation} \label{eq:div_loc_coords}
		\div^M X = \dfrac{1}{\sqrt{\det g}} \, \sum_{i = 1}^n \, \de_i \left( \sqrt{\det g} \, X_i \right) ,
	\end{equation}
	where $\det g$ is the determinant of the matrix $g = (g_{i,j})_{1 \leq i,j \leq n}$ introduced above.
	
	\subsection{The Hodge-$\star$ operator.}
	Consider an oriented Riemannian manifold $(M,g)$ and denote by $\nu$ its volume form (in local coordinates, $\nu = \sqrt{\det g} \, dx_1 \wedge \dots \wedge dx_n$). For any $k \in \mathbb{N}$, the Hodge-$\star$ operator is a linear isometry $\star: \bigwedge^k T^{\ast} M \to \bigwedge^{n-k} T^{\ast} M$ which is uniquely defined by the condition
	\begin{equation} \label{eq:id_Hodge}
		\omega \wedge (\star \psi) = \langle \omega , \psi \rangle_g \, \nu , \qquad \text{for all $k$-forms $\omega$ and $\psi$.}
	\end{equation}
	It can be shown that $\star \circ \star = \star^2$ coincides with the multiplication by $(-1)^{k(n-k)}$, and that $\star \nu = 1$.
	In particular, when $k = n-1$, the operator $\star$ allows to associate to every $(n-1)$-form $\omega \in \cA^{n-1}(M)$ a $1$-form $\star \omega \in \cA^1(M)$.
	Given a vector field $X \in \frX(M)$, denote by $i_X \nu \in \cA^{n-1}(M)$ the contraction between $X$ and the volume form $\nu$. We have the following technical Lemma (see \cite[p. 52]{Lee_book}).
	\begin{lemma} \label{lem:id_Hodge}
		Let $X \in \frX(M)$. The following identities hold:
		\begin{itemize}
			\item[(i)] %$(-1)^{n-1} X^{\flat} = \star i_X \nu$;
			$\star X^{\flat} = i_X \nu$;
			\item[(ii)] $\star d (i_X \nu) = \div^M X$.
		\end{itemize}
	\end{lemma}
	%\begin{proof}
		%See for instance \cite[p. 52]{Lee_book}.
	%\end{proof}
	
	\subsection{Locally finite perimeter sets under set operations} Given a measurable set $E \subset \R^{n+1}$ and $s \in [0,1]$, we define
	\begin{equation*}
		E^{(s)} := \left\{ (x,t) \in \R^{n+1} \, : \, \lim_{\rho \to 0^+} \dfrac{|E \cap B_{\rho}(x,t)|}{|B_{\rho}(x,t)|} = s \right\} . %\qquad \text{where} \qquad c(n+1) = |B_1|.
	\end{equation*}
	If $E$ has locally finite perimeter, we indicate by $\de^{\ast} E$ its reduced boundary and by $\nu_E$ the inner unit normal to $\de^{\ast} E$ \cite{maggi2012sets}.
	Assume now that $E$ and $F$ have locally finite perimeter. We define the set
	%\begin{equation*}
		$Q(E,F)^{\pm} := \{ (x,t) \in \de^{\ast} E \cap \de^{\ast} F \, : \, \nu_E(x,t) = \pm \nu_F(x,t) \} .$
	%\end{equation*}
	It is well known that $E \cap F$, $E \setminus F$ have also locally finite perimeter and, up to $\Hau^{n}$-negligible sets, one has (see \cite[Theorem 16.3]{maggi2012sets})
	\begin{align}
		& \de^{\ast} (E \cap F) = \left( F^{(1)} \cap \de^{\ast} E \right) \cup \left( E^{(1)} \cap \de^{\ast} F \right) \cup Q(E,F)^+ \label{eq:deEcapF} \\
		& \de^{\ast} (E \setminus F) = \left( F^{(0)} \cap \de^{\ast} E \right) \cup \left( E^{(1)} \cap \de^{\ast} F \right) \cup Q(E,F)^- , \label{eq:deEsetminusF}
	\end{align}
	and
	\begin{align}
		& \nu_{E \cap F} = \nu_E \restrict F^{(1)} + \nu_F \restrict E^{(1)} + \nu_E \restrict Q(E,F)^+ \label{eq:nuEcapF} \\
		& \nu_{E \setminus F} = \nu_E \restrict F^{(0)} - \nu_F \restrict E^{(1)} + \nu_E \restrict Q(E,F)^- . \label{eq:nuEsetminusF}
	\end{align}
	
	\section{Further notation and preliminary computations}
	\label{sec:setting}
	Let $n \geq 2$ be an integer number. For $x' \in \R^{n-1}$, we set $x = (x_1,x')$, and
	\begin{equation*}
		r := |x'| = \sqrt{x_2^2 +\dots+ x_{n}^{2}} \qquad \text{and} \qquad  \rho := |x| = \sqrt{x_1^2 + r^2} \, .
	\end{equation*}
	With this notation, the definition of $\Omega_{\lambda}$ given in the Introduction becomes
	\begin{equation*}
		\Omega_{\lambda} = \{ (x,t) \in \R^{n+1} \, : \, t > \lambda \, \rho \} .
	\end{equation*}
	We recall that $S_{\lambda} = \de \Omega_{\lambda} \setminus \{ 0 \}$ and we consider the following subsets of $\Omega_{\lambda}$ and $S_{\lambda}$, respectively:
	\begin{equation*} %\label{eq:O'S'_lambda}
		\Omega_{\lambda}' := \{ (x,t) \in \Omega_{\lambda} \, : \, x' \neq 0 \} \qquad \text{and} \qquad S_{\lambda}' := \{ (x,t) \in S_{\lambda} \, : \, x' \neq 0 \} .
	\end{equation*}
	In other words, $\Omega_{\lambda}'$ (resp. $S_{\lambda}'$) coincides with $\Omega_{\lambda}$ (resp. $S_{\lambda}$) minus the $2$-plane $x' = 0$.
	%We denote by $e_1,\dots,e_n$ the canonical basis of $\R^n$ and by $e_1,\dots,e_n,e_t$ the canonical basis of $\R^{n+1}$.
	We observe that the hypersurface $S_{\lambda}$ endowed with the Euclidean metric is a smooth Riemannian manifold. We consider the Riemannian manifold $M := \R^{n} \setminus \{ 0 \}$ equipped with the metric
	\begin{equation} \label{eq:met_g}
		g_{i,j}(x) := \delta_{i,j} + \lambda^2 \, \dfrac{x_i \, x_j}{\rho^2} , \qquad \text{for $x \in M$ and $1 \leq i,j \leq n$.}
	\end{equation}
	A direct computation yields the expression of the metric on the cotangent bundle %(see subsection \ref{subsec:Riem_Man})
	\begin{equation*}
		\tilde{g}_{i,j}(x) = \delta_{i,j} - \dfrac{\lambda^2}{1+\lambda^2} \, \dfrac{x_i \, x_j}{\rho^2} , \qquad \text{for $x \in M$ and $1 \leq i,j \leq n$.}
	\end{equation*}
	One can immediately check that $\cI : M \to S_{\lambda}$, $\cI(x) := (x,\lambda |x|)$, establishes an isometry between $S_{\lambda}$ and $M$. Moreover, $\cI$ restricts to an isometry between $M^0 := \{ x \in M \, : \, x_1 = 0 \}$ and $S^0_{\lambda} := S_{\lambda} \cap H_{\lambda}$.
	Consider the multi-index $\alpha = (2,\dots,n)$. Recalling notation \eqref{eq:multlin_not}, we introduce the $(n-2)$-form $\psi_0 \in \cA^{n-2}(M)$ defined by
	\begin{equation} \label{eq:def_psi0}
		\psi_0(x) := \dfrac{\sqrt{1 + \lambda^2}}{n-1} \, \sum_{i=2}^n \, (-1)^i \, x_i \, dx_{\alpha}^{i} , \qquad \text{for $x \in M$.}
	\end{equation}
	Additionally, denoting by $M' := \{ x \in M \, : \, x' \neq 0 \}$,
	%\begin{equation*}
	%M' := \{ x \in M \, : \, x' \neq 0 \} ,
	%\end{equation*}
	we let
	\begin{equation} \label{eq:def_theta}
		\theta := \arctan(x_1/r) \in (-\pi/2,\pi/2) , \qquad \text{for $x \in M'$.}
	\end{equation}
	In the following technical result we gather some useful identities.
	\begin{lemma}
		The following identities hold:
		\begin{align}
			& \hspace{2.6cm} \det g = 1 + \lambda^2 \label{eq:det_g} \\[0.2cm]
			& |d \rho|_g = \dfrac{1}{\sqrt{1 + \lambda^2}} \qquad |d \theta|_g = \dfrac{1}{\rho} \qquad \langle d \rho , d \theta \rangle_g = 0 . \label{eq:norm_drho_dalpha} \\[0.2cm]
			& \hspace{2cm} |\psi_0|_g = \dfrac{r}{n-1} \sqrt{1 + \lambda^2} \label{eq:norm_psi0} .
		\end{align}
	\end{lemma}
	
	\begin{proof}
		Fix $x \in M$, and, for $1 \leq i \leq n$, let
		\begin{equation*}
			\de_i = e_i + A_i \, e_t \qquad \text{where} \qquad A_i := \lambda \, \dfrac{x_i}{|x|} \, .
		\end{equation*}
		Standard properties of the determinant guarantee that $\det g = |\de_1 \wedge \dots \wedge \de_n|^2$. Define the multi-index $\beta = (1,\dots,n)$. We have
		\begin{align} \label{eq:wedge}
			\de_1 \wedge \dots \wedge \de_n & = e_{\beta} + \sum_{i=1}^{n} (-1)^{n-i} A_i \, e^{i}_{\beta} \wedge e_t .
		\end{align}
		Now we observe that the wedge products appearing in \eqref{eq:wedge} are orthonormal each other, hence
		%\begin{equation*}
			$|\de_1 \wedge \dots \wedge \de_{n}|^2 = 1+\sum_{i = 1}^nA_i^2 = 1 + \lambda^2 ,$
		%\end{equation*}
		and this proves \eqref{eq:det_g}.
		
		\noindent
		Let us now show the validity of \eqref{eq:norm_drho_dalpha}. We start observing that, by the definitions of $\rho$ and $\theta$,
		\begin{equation} \label{eq:drho_dtheta}
			d \rho = \sum_{i=1}^n \dfrac{x_i}{\rho} \, dx_i \qquad \text{and} \qquad d \theta = \dfrac{1}{\rho^2} \left( r \, dx_1 - \dfrac{x_1}{r} \sum_{i=2}^n x_i \, dx_i \right) .
		\end{equation}
		Hence we have
		\begin{equation*}
			\begin{split}
				|d\rho|_g^2 & = \sum_{i,j=1}^n \dfrac{x_i x_j}{\rho^2} \, \langle dx_i \, , \, dx_j \rangle_g \\
				& = \sum_{i,j=1}^n \dfrac{x_i x_j}{\rho^2} \left( \delta_{i,j} - \dfrac{\lambda^2}{1 + \lambda^2} \dfrac{x_i x_j}{\rho^2} \right) \\
				& = \dfrac{1}{\rho^2} \sum_{i=1}^n x_i^2 - \dfrac{\lambda^2}{\rho^4(1 + \lambda^2)} \left( \sum_{i = 1}^n x_i^2 \right) \left( \sum_{j = 1}^n x_j^2 \right) \\
				%& = \dfrac{1}{\rho^2 (1 + \lambda^2)} \sum_{i = 1}^{n} x_i^2 \\
				& = \dfrac{1}{1 + \lambda^2} ;
			\end{split}
		\end{equation*}
		\begin{equation*}
			\begin{split}
				|d \theta|_g^2 & = \dfrac{1}{\rho^4} \left( r^2 |dx_1|_g^2 - 2 x_1 \sum_{i=2}^n x_i \, \langle dx_i \, , \, dx_1 \rangle_g + \dfrac{x_1^2}{r^2} \sum_{i , j = 2}^n x_i x_j \, \langle dx_i \, , \, dx_j \rangle_g  \right) \\
				& = \dfrac{1}{\rho^4} \left( r^2 \left( 1 - \dfrac{\lambda^2}{1 + \lambda^2} \dfrac{x_1^2}{\rho^2} \right) + 2 x_1 \sum_{i=2}^n \dfrac{\lambda^2}{1 + \lambda^2} \dfrac{x_1 x_i^2}{\rho^2} + x_1^2 \left( 1 - \dfrac{\lambda^2}{1 + \lambda^2} \dfrac{r^2}{\rho^2} \right) \right) \\
				& = \dfrac{1}{\rho^4} \left( r^2 + x_1^2 \right) \\
				& = \dfrac{1}{\rho^2} ;
			\end{split}
		\end{equation*}
		\begin{equation*}
			\begin{split}
				\langle d \rho \, , \, d \theta \rangle_g & = \dfrac{1}{\rho^2} \left\langle \sum_{i=1}^n \dfrac{x_i}{\rho} \, dx_i \, , \, r \, dx_1 - \dfrac{x_1}{r} \sum_{i=2}^n x_i \, dx_i \right\rangle_g \\
				& = \dfrac{1}{\rho^2} \left( \dfrac{r}{\rho} \, x_1 \, |dx_1|_g^2 +\left( \dfrac{r}{\rho} - \dfrac{x_1^2}{r \rho} \right)\sum_{i=2}^n x_i \, \langle dx_1 \, , \, dx_i \rangle_g - \dfrac{x_1}{r \rho} \sum_{i,j=2}^n x_i x_j \, \langle dx_i \, , \, dx_j \rangle_g \right) \\
				& = \dfrac{1}{\rho^2} \left( \dfrac{r}{\rho} \, x_1 \left( 1 - \dfrac{\lambda^2}{1 + \lambda^2} \dfrac{x_1^2}{\rho^2} \right) - \dfrac{\lambda^2}{1+\lambda^2} \left( \dfrac{r}{\rho} - \dfrac{x_1^2}{r \rho} \right) \sum_{i=2}^{n} \dfrac{x_1 x_i}{\rho^2} - \dfrac{x_1}{\rho} \, r \left( 1 - \dfrac{\lambda^2}{1 + \lambda^2} \dfrac{r^2}{\rho^2} \right) \right) \\
				& = \dfrac{1}{\rho^2} \left( \dfrac{r}{\rho} x_1 \left( 1 - \dfrac{\lambda^2}{1 + \lambda^2} \dfrac{x_1^2}{\rho^2} -  \right) - \dfrac{\lambda^2}{1+\lambda^2} \left( \dfrac{r}{\rho} - \dfrac{x_1^2}{r \rho} \right) x_1 \, \dfrac{r^2}{\rho^2} - \dfrac{x_1}{\rho} \, r \left( 1 - \dfrac{\lambda^2}{1 + \lambda^2} \dfrac{r^2}{\rho^2} \right) \right) \\[0.2cm]
				& = 0 .
			\end{split}
		\end{equation*}
		
		\noindent
	    Ultimately, we prove \eqref{eq:norm_psi0}.
		We have
		\begin{equation} \label{eq:norm_psi0_pf}
			\begin{split}
				|\psi_0|_g^2 & = \langle \psi_0 , \psi_0 \rangle_g \\
				& = \dfrac{1 + \lambda^2}{(n-1)^2} \left\langle \sum_{i = 2}^n (-1)^i \, x_i \, dx_{\alpha}^i \, , \sum_{j = 2}^n (-1)^j \, x_j \, dx_{\alpha}^j \right\rangle_{g} \\
				& = \dfrac{1 + \lambda^2}{(n-1)^2} \sum_{i,j = 2}^n (-1)^{i+j} \, x_i \, x_j \, \langle dx_{\alpha}^i \, , dx_{\alpha}^j \rangle_g .
			\end{split}
		\end{equation}
		Denote by $\tilde{g}_1$ the $(n-1)\times(n-1)$ submatrix of $\tilde{g}$ obtained by cancelling the first row and the first column of $\tilde{g}$.
		By the definition of the metric on $\bigwedge^k T^{\ast} M$ given in \eqref{eq:metric_forms} (in this case $k = n - 2$) and the properties of the adjugate matrix (see \cite[p. 73]{Shafarevich_Remizov_book}), we have
		\begin{equation} \label{eq:scalar_n-2_form}
			\begin{split}
				\langle dx_{\alpha}^i \, , dx_{\alpha}^j \rangle_g = (-1)^{i+j} \, (\tilde{g}_1^{-1})_{i,j} \, \det (\tilde{g}_1) .
			\end{split}
		\end{equation}
		A mere computation yields
		\begin{equation} \label{eq:term_g[1,1]}
			(\tilde{g}_1^{-1})_{i,j} = \delta_{i,j} + \dfrac{\lambda^2}{\rho^2 + \lambda^2 \, x_1^2} \, x_i \, x_j ,
		\end{equation}
		The same calculation performed above to compute $\det g$ guarantees also that
		\begin{equation} \label{eq:det_g[1,1]}
			\det(\tilde{g}_1)^{-1} = \det (\tilde{g}_1^{-1}) = 1 + \dfrac{\lambda^2 \, r^2}{\rho^2 + \lambda^2 \, x_1^2} . %= \dfrac{(1 + \lambda^2) \, \rho^2}{\rho^2 + \lambda^2 \, x_1^2} .
		\end{equation}
		Consequently, inserting \eqref{eq:det_g[1,1]} and \eqref{eq:term_g[1,1]} inside \eqref{eq:scalar_n-2_form}, by \eqref{eq:norm_psi0_pf}, we infer that
		\begin{equation*}
			\begin{split}
				|\psi_0|_g^2 & = \dfrac{1 + \lambda^2}{n-1} \sum_{i,j = 2}^n \, x_i \, x_j \, \left( \delta_{i,j} + \dfrac{\lambda^2}{\rho^2 + \lambda^2 \, x_1^2} \, x_i \, x_j \right) \left( 1 + \dfrac{\lambda^2 \, r^2}{\rho^2 + \lambda^2 \, x_1^2} \right)^{-1} %\dfrac{\rho^2 + \lambda^2 \, x_1^2}{(1 + \lambda^2) \, \rho^2}
				\\
				& = \dfrac{1 + \lambda^2}{n-1} \, r^2 ,
			\end{split}
		\end{equation*}
		that is precisely \eqref{eq:norm_psi0}.
	\end{proof}

\section{Construction of the calibration}
\label{sec:constr_cal}
The main result of this section is the following Theorem.
\begin{theorem} \label{thm:excal}
	For $n \geq 4$, consider the threshold $\bar{\lambda}(n)$ defined in \eqref{eq:def_barlambda}.
	Then, if $0<\lambda \leq \bar \lambda(n)$, there exists a vector field $Z : \Omega_{\lambda}' \cup S_{\lambda}' \to \R^{n+1}$ such that:
	\begin{itemize}
		\item[(i)]
		$Z\left( x,t \right) = Z \left(x, \lambda |x| \right)$, for all $(x,t) \in \Omega_{\lambda}' \cup S_{\lambda}'$;
		\medskip
		\item[(ii)] $|Z(x,t)| \leq 1$, for all $(x,t) \in \Omega_{\lambda}' \cup S_{\lambda}'$;
		\medskip
		\item[(iii)] $Z(x,t) = e_1$, for all $(x,t) \in \overline{H}_{\lambda} \cap (\Omega_{\lambda}' \cup S_{\lambda}')$;
		\medskip
		\item[(iv)] $Z (x,t) \in T_{(x,t)} S_{\lambda}$, for all $(x,t) \in S_{\lambda}'$;
		\medskip
		\item[(v)] %$Z \in C^{\infty}(\Omega_{\lambda}';\R^n)$ 
		$Z$ is smooth and $\div Z = 0$ in $\Omega_{\lambda}'$.
	\end{itemize}
\end{theorem}

A key passage for the proof of Theorem \ref{thm:excal}, is the construction of a calibration for the $(n-1)$-dimensional surface $S^0_{\lambda}$ inside $S_{\lambda}$. This is the subject of the next result.

\begin{theorem} \label{thm:excal_Slambda}
	For $n \geq 4$, let $0 < \lambda \leq \bar{\lambda}(n)$. Then there exists a vector field $Y \in \frX(S_{\lambda}')$ satisfying the following conditions:
	\begin{itemize}
		\item[(i)] $|Y(p)| \leq 1$, for all $p \in S_{\lambda}'$;
		\medskip
		\item[(ii)] $Y(p) = e_1$, for all $p \in S^0_{\lambda}$;
		\medskip
		\item[(iii)] $Y$ is smooth and $\div^{S_{\lambda}} Y = 0$ on $S_{\lambda}'$.
	\end{itemize}
\end{theorem}

The proof of Theorem \ref{thm:excal_Slambda} can be in part found in \cite{morgan2002area}.
For the sake of clarity, we provide a complete proof here.

\begin{proof}
	Let $\alpha$ be the multi-index $(2,\dots,n)$, and consider the $(n-1)$-form $\omega_0 = \sqrt{1 + \lambda^2} \, dx_{\alpha}$. A direct computation shows that, in $M'$,
	\begin{equation} \label{eq:omega_0}
		\omega_0 = d \psi_0 = \dfrac{n-1}{r} \, dr \wedge \psi_0 , %\qquad \text{in $M'$,}
	\end{equation}
	where $\psi_0$ is as in \eqref{eq:def_psi0}. Additionally, set $u := x_1/r$. Given a differentiable function $h$, we define
	\begin{equation*}
		\psi_h(x) := h(u) \, \psi_0(x) \qquad \text{and} \qquad \omega_h(x) := d\psi_h(x) , \qquad \text{for $x \in M'$.} 
	\end{equation*}
	By the properties of the exterior derivative, we have
	\begin{equation} \label{eq:omega_h}
		\omega_{h} = d h \wedge \psi_0 + h \, \omega_0 \, ,
	\end{equation}
	where
	\begin{equation} \label{eq:dh}
		dh = \dfrac{h'(u)}{r} \, dx_1 - \dfrac{x_1}{r^2} \, h'(u) \, dr \, .
	\end{equation}
	\\
	Putting together \eqref{eq:omega_h}, \eqref{eq:dh} and using \eqref{eq:omega_0}, we infer that
	\begin{align} \label{eq:scompomega_h}
		\omega_h(x) & = \left( \dfrac{h'(u)}{r} \, dx_1 - \dfrac{x_1}{r^2} \, h'(u) \, dr \right) \wedge \psi_0(x) + h(u) \, \omega_0(x) \nonumber \\
		& = \left[ \dfrac{r}{n-1} \left( \dfrac{h'(u)}{r} \, dx_1 - \dfrac{x_1}{r^2} \, h'(u) \, dr \right) + h(u) \, dr \right] \wedge \left( \dfrac{n-1}{r} \, \psi_0(x) \right) \\
		& = \left[ \left( h(u) - \dfrac{x_1}{(n-1) \, r} h'(u) \right) dr + \dfrac{h'(u)}{n-1} \, dx_1 \right] \wedge \left( \dfrac{n-1}{r} \, \psi_0(x) \right) \, . \nonumber
	\end{align}
	
	\noindent
	\textbf{Step 1.}
	We look for a continuously differentiable function $h : \R \to \R$ such that:
	\begin{equation} \label{eq:req_h}
		h(0) = 1 , \qquad h'(0) = 0 \qquad \text{and} \qquad
		|\omega_{h}(x)|_g \leq 1 \quad \text{for all $x \in M'$.}
	\end{equation}
	Applying the Cauchy-Schwartz inequality for wedge products (see \cite[pag. 32]{Federer_book}, and note that $1$-forms are always simple), by \eqref{eq:scompomega_h}, we obtain
	\begin{equation*}
		|\omega_h(x)|_g \leq \left| \left( h(u) - \dfrac{x_1}{(n-1) r} h'(u) \right) dr + \dfrac{h'(u)}{n-1} dx_1 \right|_g \cdot \, \, \left| \dfrac{n-1}{r} \psi_0(x) \right|_g \, .
	\end{equation*}
	Thanks to \eqref{eq:norm_psi0}, to have $|\omega_h|_g \leq 1$, it suffices to impose
	\begin{equation} \label{eq:condomega_hleq1}
		\left| \left( h(u) - \dfrac{x_1}{(n-1) \, r} h'(u) \right) dr + \dfrac{h'(u)}{n-1} \, dx_1 \right|_g^2 \leq \dfrac{1}{1 + \lambda^2} \, .
	\end{equation}
	We now observe that, taking $\theta \in (-\pi/2,\pi/2)$ as in \eqref{eq:def_theta}, one has
	\begin{equation*}
		u = \tan \theta , \qquad x_1 = \rho \sin \theta \qquad \text{and} \qquad r = \rho \cos \theta .
	\end{equation*}
	The identities \eqref{eq:norm_drho_dalpha} permit us to rewrite \eqref{eq:condomega_hleq1} in the following way:
	\begin{equation} \label{eq:condomega_hleq1_rew1}
		h(u)^2 + (1 + \lambda^2) \left[ \dfrac{(1 + u^2) \, h'(u)}{n-1} - u \, h(u) \right]^2 \leq 1 + u^2 \, .
	\end{equation}
	Thus we are reduced to find a solution $h$ to the Cauchy problem for the differential inequality \eqref{eq:condomega_hleq1_rew1} with initial conditions $h(0) = 1$, $h'(0) = 0$.
	In order to solve \eqref{eq:condomega_hleq1_rew1}, we apply the ansatz
	\begin{equation*}
		h(u) = \dfrac{\cos (\beta(\arctan u))}{\cos (\arctan u)} = \dfrac{\cos (\beta(\theta))}{\cos \theta} \, ,
	\end{equation*}
	\\
	where $\beta: (-\pi/2,\pi/2) \to (-\pi/2,\pi/2)$ is a continuously differentiable function with $\beta(0) = 0$. By this approach, $h(0) = 1$. Moreover, we have
	\begin{equation} \label{eq:h'}
		h'(u) = \dfrac{ -\sin(\beta (\theta)) \beta'(\theta) \cos \theta + \cos(\beta(\theta)) \sin \theta }{(1 + u^2) \cos^2(\theta)} \, ,
	\end{equation}
	and then, in particular, $h'(0) = 0$. We can now rewrite the differential inequality \eqref{eq:condomega_hleq1_rew1} in terms of $\beta$ and $\theta$. Thanks to \eqref{eq:h'}, we have
	\begin{align*}
		\dfrac{(1 + u^2) \, h'(u)}{n-1} - u \, h(u) & = \dfrac{ -\sin (\beta (\theta)) \beta'(\theta) \cos \theta + \cos(\beta(\theta)) \sin \theta }{(n-1) \cos^2(\theta)} - \tan (\theta) \dfrac{\cos (\beta(\theta))}{\cos \theta} \\
		& = \dfrac{1}{n-1} \left( - \beta'(\theta) \dfrac{\sin(\beta (\theta))}{\cos \theta} + \dfrac{\tan \theta}{\tan(\beta(\theta))} \dfrac{\sin(\beta (\theta))}{\cos \theta} \right) - \dfrac{\tan \theta}{\tan(\beta(\theta))} \dfrac{\sin(\beta(\theta))}{\cos \theta} \\
		& = \dfrac{\sin(\beta(\theta))}{\cos \theta} \left( - \dfrac{\beta'(\theta)}{n-1} + \dfrac{2-n}{n-1} \dfrac{\tan \theta}{\tan(\beta(\theta))} \right) \, .
	\end{align*}
	Thus, by the definition of $h$ and the identity $\cos^{-2}(\theta) = 1 + u^2$, \eqref{eq:condomega_hleq1_rew1} becomes
	\begin{equation} \label{eq:condomega_hleq1_rew2}
		\left| \beta'(\theta) + (n-2) \dfrac{\tan \theta}{\tan(\beta(\theta))} \right| \leq \dfrac{n-1}{\sqrt{1 + \lambda^2}} , \qquad \text{for $\theta \in (-\pi/2,\pi/2)$.}
	\end{equation}
	We want to find a solution $\beta$ to the above differential inequality with initial condition $\beta(0) = 0$. We claim 
	%that a good choice is $\bar \beta = \beta_{\bar{\gamma}}$, for a suitable $\bar{\gamma} > 0$, where 
	that such a solution can be found among the functions $\beta_{\gamma} : (-\pi/2,\pi/2) \to (-\pi/2,\pi/2)$ defined by
	\begin{equation*} %\label{eq:beta_gamma}
		\beta_{\gamma}(\theta) := \sgn (\theta) \arccos \left( \cos^{1 + \gamma} (\theta) \right) \qquad \text{for $\gamma > 0$,}
	\end{equation*}
	i.e., that there exists $\gamma > 0$ such that $\beta_{\gamma}$ satisfies \eqref{eq:condomega_hleq1_rew2} and $\beta_{\gamma}(0) = 0$.
	%We observe that the definition of $\beta_{\gamma}$ is well-posed, because $\cos \theta > 0$, for all $\theta \in (-\pi/2,\pi/2)$. In addition, for the same reason, $\beta_{\gamma}(\theta) \in (-\pi/2,\pi/2)$, for all $\theta \in (-\pi/2,\pi/2)$. 
	%Since the initial condition is trivially satisfied, we only look for $\gamma > 0$ such that $\beta_{\gamma}$ solves \eqref{eq:condomega_hleq1_rew2}. We observe that $\beta_{\gamma}$ is differentiable in $(-\pi/2,\pi/2)$, and
	Note that, for any $\gamma > 0$, $\beta_{\gamma}(0) = 0$. Let us rewrite \eqref{eq:condomega_hleq1_rew2} for $\beta_{\gamma}$.
	We observe that $\beta_{\gamma}$ is differentiable, and
	\begin{equation*}
		\beta_{\gamma}'(\theta) = \sgn (\theta) \dfrac{(1 + \gamma) \cos^{\gamma} \theta \, \sin \theta}{\sqrt{1 - \cos^{2(1 + \gamma)}(\theta)}} \, .
	\end{equation*}
	Moreover, we have
	\begin{equation*}
		\dfrac{1}{\tan (\beta_\gamma(\theta))} = \dfrac{\cos (\beta_{\gamma}(\theta))}{\sgn(\beta_{\gamma}(\theta)) \sqrt{1 - \cos^2 (\beta_{\gamma}(\theta))}} = \dfrac{\sgn(\theta) \cos^{1 + \gamma}(\theta)}{\sqrt{1 - \cos^{2(1 + \gamma)}(\theta)}} ,
	\end{equation*}
	and consequently,
	\begin{equation*}
		\left| \beta'(\theta) + (n-2) \dfrac{\tan \theta}{\tan(\beta(\theta))} \right| = (n-1+\gamma) \left| \dfrac{\tan \theta}{\tan(\beta_{\gamma}(\theta))} \right| . %\leq \dfrac{n-1}{\sqrt{1 + \lambda^2}} .
	\end{equation*}
	Hence, for this family of functions $\beta_{\gamma}(\theta)$, \eqref{eq:condomega_hleq1_rew2} transforms into
	\begin{equation} \label{eq:condomega_hleq1_rew2_beta}
		\left( 1 + \dfrac{\gamma}{n-1} \right)^2 \tan^2 \theta \leq \dfrac{\tan^2 (\beta_{\gamma}(\theta))}{1 + \lambda^2} \, , \qquad \text{for $\theta \in (-\pi/2,\pi/2)$.}
	\end{equation}
	Let us set $z := \cos^{-2} (\theta)$. In this way, $\cos^{-2}(\beta_{\gamma}(\theta)) = z^{1 + \gamma}$, and condition \eqref{eq:condomega_hleq1_rew2_beta} becomes
	\begin{equation} \label{eq:condomega_hleq1_rew2_zeta}
		\left( 1 + \dfrac{\gamma}{n-1} \right)^2 \leq \dfrac{1}{1 + \lambda^2} \, \dfrac{z^{1 + \gamma} - 1}{z - 1} \, , \qquad \text{for $z \geq 1$.}
	\end{equation}
	Let $w(z) := (z^{1 + \gamma} - 1)/(z - 1)$. For $z \geq 1$, $w$ is monotonically increasing, hence it suffices to test condition \eqref{eq:condomega_hleq1_rew2_zeta} only as $z \to 1^+$. Since $\lim_{z \to 1^+} w(z) = 1 + \gamma$, we get the condition
	\begin{equation} \label{eq:condomega_hleq1_rew2_zlim}
		\left( 1 + \dfrac{\gamma}{n-1} \right)^2 - \dfrac{1 + \gamma}{1 + \lambda^2} \, \leq \, 0 \, .
	\end{equation}
	In particular, \eqref{eq:condomega_hleq1_rew2_zlim} implies \eqref{eq:condomega_hleq1_rew2_zeta}. We now observe that the function on the LHS of \eqref{eq:condomega_hleq1_rew2_zlim} has a global minimizer in
	\begin{equation*}
		\bar{\gamma} := (n-1) \left( \dfrac{n-1}{2(1 + \lambda^2)} - 1 \right) .
	\end{equation*}
	Let $\gamma = \bar{\gamma}$. If \eqref{eq:condomega_hleq1_rew2_zlim} is satisfied, by construction, the function $\beta = \beta_{\gamma}$ solves \eqref{eq:condomega_hleq1_rew2} and $\beta(0) = 0$.
	A mere computation gives that \eqref{eq:condomega_hleq1_rew2_zlim} holds for $\gamma = \bar{\gamma}$ if and only if
	\begin{equation} \label{eq:condlambda}
		n \geq 4 \qquad \text{and} \qquad 0 < \lambda \leq \bar{\lambda}(n) := \dfrac{1}{2} \dfrac{n-3}{\sqrt{n-2}} \, ,
	\end{equation}
	and this concludes the proof of Step 1.
	In particular, when \eqref{eq:condlambda} is fulfilled, the function $h$ for which \eqref{eq:req_h} holds true is given by
	\begin{align*}
		h(u) %& = \dfrac{\cos(\beta(\arctan u))}{\cos(\arctan u)} \\
		%& = \dfrac{\cos(\sgn(\theta) \arccos(\cos^{1 + \gamma}(\arctan u)))}{\cos(\arctan u)} \\
		& = \cos^{\gamma}(\arctan u), \qquad \text{where} \qquad \gamma = (n-1) \left( \dfrac{n-1}{2(1 + \lambda^2)} - 1 \right) .
	\end{align*}
    Notice that $h$ is smooth, and so $\omega_{h} \in \cA^{n-1}(M)$ is smooth.
	
	\smallskip
	\noindent
	\textbf{Step 2.} For $x \in M'$, define
	\begin{equation*}
		X = (-1)^{n-1}(\star \, \omega_{h})^{\#} \in \frX(M') \qquad \text{and} \qquad Y = \cI_{\ast} X \in \frX(S_{\lambda}') .
	\end{equation*}
	We now show that $Y$ satisfies (i), (ii) and (iii).
	By construction, $|\omega_{h}|_g \leq 1$, and since $\star$ and $\cI$ are isometries, we trivially infer that $|Y(p)| \leq 1$, for all $p \in S_{\lambda}'$, that is (i).
	Moreover, by \eqref{eq:req_h} and \eqref{eq:scompomega_h}, since $u = 0$ as $x_1 = 0$ , we have $\omega_{h} = \omega_0 = \sqrt{1 + \lambda^2} \, dx_2\wedge\dots\wedge dx_n$ on $M^0$. It is easy to check that, by \eqref{eq:id_Hodge} (note that $\nu = \sqrt{\det g} \, dx_1 \wedge \dots \wedge dx_n = \sqrt{1+\lambda^2} \, dx_1 \wedge \dots \wedge dx_n$), the following equality holds
	\begin{equation*}
		\star \left(\sqrt{1 + \lambda^2} \, dx_2 \wedge \dots \wedge dx_n\right) = (-1)^{n-1} e_1^{\flat} ,
	\end{equation*}
	hence $X(x) = e_1$, for every $x \in M^0$. By the definition of $\cI$, we have $Y(p) = e_1$, for all $p \in S^0_{\lambda}$, that is precisely (ii).
	Finally, let us prove (iii). The smoothness of $Y$ immediately follows from that of $\omega_{h}$. Now, by (i) of Lemma \ref{lem:id_Hodge}, $\star X^{\flat} = i_X \nu$,
	hence we have
	\begin{equation*}
		i_X \nu = (-1)^{n-1} \star^2 \omega_{h} = \omega_{h}.
	\end{equation*}
	By (ii) of Lemma \ref{lem:id_Hodge} and the definition of $\omega_h$, we deduce that
	\begin{equation*}
		\div^M X = \star d \omega_h = \star d d \psi_{h} = 0 ,
	\end{equation*}
	and then (iii) immediately follows from \eqref{eq:prop_divM}.
\end{proof}

Let us now proceed with the proof of Theorem \ref{thm:excal}. The idea is to extend the vector field $Y$ of Theorem \ref{thm:excal_Slambda} to $\Omega_{\lambda}'$ by vertically projecting onto $S_{\lambda}'$. This operation will produce a divergence-free vector field in $\Omega_{\lambda}'$.

\begin{proof}[Proof (of Theorem \ref{thm:excal}).]
	Consider the vector field
	\begin{equation*}
		Z(x,t) := Y(x,\lambda|x|) \, .
	\end{equation*}
	Since $Y$ is defined on $S_{\lambda}'$, it is clear that $Z$ is defined in $\Omega_{\lambda}'$. By construction, (i) of Theorem \ref{thm:excal} holds, while the statements (ii), (iii), (iv) immediately follow from the properties of $Y$. It remains to demonstrate (v). For sure, $Z$ is smooth because $Y$ is. We need to show that $\div Z = 0$ in $\Omega_{\lambda}'$. Consider the (non-orthonormal) frame $\de_1,\dots,\de_n,\de_t$ for $T \Omega_{\lambda}'$ defined by
	\begin{equation*}
		\de_i(x,t) = \de_i(x) = e_i + \lambda \, \dfrac{x_i}{|x|} \, e_t \qquad \text{for $1 \leq i \leq n$} \qquad \text{and} \qquad \de_t = e_t .
	\end{equation*}
	In addition, let $\hat{g}$ the matrix whose entries are the scalar products $\langle \de_i , \de_j \rangle$, $ \langle \de_i , \de_t \rangle $, for $1 \leq i,j \leq n$. By \eqref{eq:wedge}, the determinant of $\hat{g}$ satisfies
	\begin{align*}
		\det \hat g & = |\de_1 \wedge \dots \wedge \de_n \wedge e_t|^2 \\
		& = |e_1 \wedge \dots \wedge e_n \wedge e_t|^2 \\
		& = 1 ,
	\end{align*}
	%Let us represent $Y$ with respect to $\de_1,\dots,\de_n$ as
	%\begin{equation*}
		%Y = \sum_{i=1}^n Y_i \de_i .
	%\end{equation*}
	and, in particular, is constant. Represent now $Y$ and $Z$ in coordinates with respect to $\de_1,\dots,\de_n$ and $\de_1,\dots,\de_n,\de_t$ respectively. Then, for suitable $Y_i \in C^{\infty}(S_{\lambda}')$, $Z_i, Z_t \in C^{\infty}(\Omega_{\lambda}')$, 
	\begin{equation*}
	Y(x,\lambda|x|) = \sum_{i=1}^n Y_i(x,\lambda|x|) \de_i(x) \qquad \text{and} \qquad Z(x,t) = \sum_{i=1}^n Z_i(x,t) \de_i(x) + Z_t(x,t) \de_t , 
	\end{equation*}
    for all $x \in \R^n$ and $t \geq \lambda|x|$. By construction, we have
	\begin{equation*}
		Z_i(x,t) = Y_i(x,\lambda|x|) \quad \text{for any $1 \leq i \leq n$} \qquad \text{and} \qquad Z_t \equiv 0.
	\end{equation*}
	Consequently, since by \eqref{eq:det_g} also $\det g = 1 + \lambda^2$ is constant, \eqref{eq:div_loc_coords} ensures that
	\begin{equation*}
		\div Z (x,t) = \sum_{i=1}^n \de_i Z_i(x,t) = \sum_{i = 1}^n \de_i Y_i (x,\lambda|x|)= \div^{S_{\lambda}'} Y (x,\lambda|x|) = 0 .
	\end{equation*}
    This closes the proof.
\end{proof}

\section{Minimality of $E$}
\label{sec:Min_E}
In this section, we demonstrate Theorem \ref{thm:min}. The idea is to apply the Divergence Theorem to the vector field $Z$ in the intersection between $E \difsim F$ and a sequence of sets invading $\Omega_{\lambda}' \subset \Omega_{\lambda}$. Subsequently, since $Z$ is tangent to $S_{\lambda}'$, by an approximation argument, we will be able to infer the minimality condition of Theorem \ref{thm:min}.

\begin{proof}[Proof (of Theorem \ref{thm:min}).]
    Up to rescaling, we can assume $R = 1$. For $\epsilon > 0$, let
	\begin{equation*}
		A^{\epsilon} := \{ (x,t) \in \R^{n+1} \, : \, |x'| \leq \epsilon \} \qquad \text{and} \qquad \Omega_{\lambda}^{\epsilon} := \Omega_{\lambda} \setminus A^{\epsilon} .
	\end{equation*}
	We observe that, for every $\epsilon > 0$, $\Omega_{\lambda}^{\epsilon} \subset \Omega_{\lambda}'$. When $0 < \lambda \leq \bar{\lambda}(n)$, we can consider the vector field $Z : \Omega_{\lambda}' \to \R^{n+1}$ of Theorem \ref{thm:excal}. By the Divergence Theorem (see \cite{maggi2012sets}) and since $Z$ is divergence-free, we have
	\begin{equation} \label{eq:0=divZ}
		0 = \int_{(F \setminus E)\cap\Omega_{\lambda}^{\epsilon}} \div Z = \int_{\de^{\ast}((F \setminus E)\cap\Omega_{\lambda}^{\epsilon})} \left\langle Z \, , \, \nu_{(F \setminus E)\cap\Omega_{\lambda}^{\epsilon}} \right\rangle \, d \Hau^{n-1} .
	\end{equation}
	Owing to \eqref{eq:deEcapF}, there holds
	\begin{equation*}
		\de^{\ast} ((F \setminus E) \cap \Omega_{\lambda}^{\epsilon}) = \left( \Omega_{\lambda}^{\epsilon} \cap \de^{\ast}(F \setminus E)  \right) \cup \left( (F \setminus E)^{(1)} \cap \de\Omega_{\lambda}^{\epsilon} \right) \cup Q\left(F \setminus E,\Omega_{\lambda}^{\epsilon}\right)^+ ,
	\end{equation*}
	and this in combination with \eqref{eq:0=divZ} and \eqref{eq:nuEcapF} ensures that
	\begin{equation} \label{eq:0=divZ_2}
		\begin{split}
			0 = \int_{(F \setminus E)\cap\Omega_{\lambda}^{\epsilon}} \div Z & = \int_{\Omega_{\lambda}^{\epsilon} \cap \de^{\ast}(F \setminus E)} \left\langle Z \, , \, \nu_{F \setminus E} \right\rangle \, d \Hau^{n-1} + \\
			& \qquad \qquad + \int_{\left( (F \setminus E)^{(1)} \cap \de\Omega_{\lambda}^{\epsilon} \right) \, \cup \, Q\left(F \setminus E,\Omega_{\lambda}^{\epsilon}\right)^+} \left\langle Z \, , \, \nu_{\Omega_{\lambda}^{\epsilon}} \right\rangle \, d \Hau^{n-1} \\
			& = (I) + (II) .
		\end{split}
	\end{equation}
	By \eqref{eq:deEsetminusF}, because $Z = e_1 = \nu_E$ on $\de^{\ast} E \cap \Omega_{\lambda} = H_{\lambda}$, one has
	\begin{equation} \label{eq:est_(I)}
		\begin{split}
			(I) & = \int_{\Omega_{\lambda}^{\epsilon} \cap E^{(0)} \cap \de^{\ast} F} \left \langle Z \, , \, \nu_{F} \right \rangle \, d \Hau^{n-1} - \int_{\Omega_{\lambda}^{\epsilon} \cap F^{(1)} \cap \de^{\ast} E} \left \langle Z \, , \, \nu_{E} \right \rangle \, d \Hau^{n-1} - \int_{\Omega_{\lambda}^{\epsilon} \, \cap \, Q(E,F)^-} \left \langle Z \, , \, \nu_{E} \right \rangle \, d \Hau^{n-1} \\[0.2cm]
			& \leq \p(F;E^{(0)} \cap \Omega_{\lambda}^{\epsilon}) - \p(E;F^{(1)} \cap \Omega_{\lambda}^{\epsilon}) \\
			& = \p(F;\bar{E}^{c} \cap \Omega_{\lambda}^{\epsilon}) - \p(E;F^{(1)} \cap \Omega_{\lambda}^{\epsilon}) .
		\end{split}
	\end{equation}
	Let us now estimate $(II)$. Because $E \difsim F \Subset B_1$, clearly $\de^{\ast} (F \setminus E) \subset B_1$ and $(F \setminus E)^{(1)} \subset B_1$. Thus, by applying \eqref{eq:deEsetminusF} to $\Omega_{\lambda}^{\epsilon} = \Omega_{\lambda} \setminus A^{\epsilon}$ and \eqref{eq:nuEsetminusF} to its normal, we have
	\begin{equation} \label{eq:est_(II)}
		\begin{split}
			|(II)| & \leq \int_{B_1 \cap \de \Omega_{\lambda}^{\epsilon}} \left| \left \langle Z \, , \, \nu_{\Omega_{\lambda}^{\epsilon}} \right \rangle \right| \, d \Hau^{n-1} \\
			& \leq \int_{B_1 \cap A^{\epsilon} \cap \de \Omega_{\lambda}} \left| \left \langle Z \, , \, \nu_{\Omega_{\lambda}} \right \rangle \right| \, d \Hau^{n-1} + \int_{B_1 \cap \Omega_{\lambda} \cap \de A^{\epsilon}} \left| \left \langle Z \, , \, \nu_{A^{\epsilon}} \right \rangle \right| \, d \Hau^{n-1} .
		\end{split}
	\end{equation}
	Now, since $Z$ is tangent to $S_{\lambda}'$, the first integral on the right hand side of \eqref{eq:est_(II)} vanishes. This implies that
	\begin{equation*}
		|(II)| \leq \Hau^{n-1}(B_1 \cap \de A^{\epsilon}) = C(n) \, \epsilon^{n-2} .
	\end{equation*}
    This in combination with \eqref{eq:0=divZ_2} and \eqref{eq:est_(I)} yields
    \begin{equation} \label{eq:sum_1_calib}
    	\p(E;F^{(1)} \cap \Omega_{\lambda}^{\epsilon}) \leq \p(F;\bar E ^ c \cap \Omega_{\lambda}^{\epsilon}) + C(n) \, \epsilon^{n-2} .
    \end{equation}
    By applying the Divergence Theorem in $(E \setminus F) \cap \Omega_{\lambda}^{\epsilon}$ and arguing very similarly as above, we also obtain
    \begin{equation} \label{eq:sum_2_calib}
    	\begin{split}
    		\p(E;F^{(0)} \cap \Omega_{\lambda}^{\epsilon}) & \leq \p(F;E^{(1)} \cap \Omega_{\lambda}^{\epsilon}) + C(n) \, \epsilon^{n-2} \\
    		& = \p(F;E \cap \Omega_{\lambda}^{\epsilon}) + C(n) \, \epsilon^{n-2}
    	\end{split}
    \end{equation}
    Summing up \eqref{eq:sum_1_calib} and \eqref{eq:sum_2_calib}, noticing that
    \begin{equation*}
    	\p(E;F^{(1)} \setminus B_1) = \p(E;F^{(0)} \setminus B_1) = 0 \qquad \text{and} \qquad \p(F;E \setminus B_1) = \p(F;\bar{E}^{c}\setminus B_1) = 0 ,
    \end{equation*}
    we deduce that
    \begin{equation*}
    	\p(E;(F^{(0)} \cup F^{(1)}) \cap \Omega_{\lambda}^{\epsilon} \cap B_1) \leq \p(F;(\bar{E}^c \cup E) \cap \Omega_{\lambda}^{\epsilon} \cap B_1) + C(n) \, \epsilon^{n-2} .
    \end{equation*}
    %By Federer's Theorem \cite[Theorem 16.2]{maggi2012sets}, we get
    %\begin{equation*}
    	%E^{(0)} \cup E^{(1)} \cup E^{(1/2)} = F^{(0)} \cup F^{(1)} \cup F^{(1/2)} = \R^{n+1} \qquad \text{and} \qquad \p(F;E^{})
    %\end{equation*}
    By applying Federer's Theorem \cite[Theorem 16.2]{maggi2012sets}, we then get
    \begin{equation*}
    	\p(E;\Omega_{\lambda}^{\epsilon} \cap B_1) \leq \p(F;\Omega_{\lambda}^{\epsilon} \cap B_1) + C(n) \, \epsilon^{n-2} .
    \end{equation*}
    Finally, passing to the limit as $\epsilon \to 0^+$ in the relation above, we obtain that $$\p(E;\Omega_{\lambda} \cap B_1) \leq \p(F;\Omega_{\lambda} \cap B_1),$$ and we conclude.
\end{proof}

%\bibliography{biblio_cal}{}
%\bibliographystyle{plain}

\end{document}